\documentclass[12pt,leqno,twoside,sort]{amsart}

\usepackage{amssymb,amsmath,amsthm,soul,color}
\usepackage[normalem]{ulem}
\usepackage{amssymb,amsmath,amsthm}
\usepackage{mathrsfs}
\usepackage[utf8]{inputenc}
\usepackage{todonotes}
\usepackage[left=2cm, right=2cm, top=2cm, bottom=2cm]{geometry}
\usepackage{url}
\usepackage[numbers,sort&compress]{natbib}
\usepackage{fixmath}

\usepackage[T1]{fontenc}
\usepackage{lmodern}
\usepackage{microtype}

\usepackage[colorlinks=true,urlcolor=blue,
citecolor=red,linkcolor=blue,linktocpage,pdfpagelabels,
bookmarksnumbered,bookmarksopen]{hyperref}
\usepackage{mathtools}
\usepackage{xcolor,cancel}

\newtheorem{Th}{Theorem}[section]
\newtheorem{Prop}[Th]{Proposition}
\newtheorem{Lem}[Th]{Lemma}

\newtheorem{Rem}[Th]{Remark}

\newcommand{\eps}{\varepsilon}

\newcommand{\R}{\mathbb{R}}

\newcommand{\cJ}{{\mathcal J}}

\newcommand{\weakto}{\rightharpoonup}

\renewcommand{\div}{\mathrm{div}\,}

\numberwithin{equation}{section}

\newcommand{\Lg}{\Delta_g}

\begin{document}

\title{Note on elliptic equations on closed manifolds with singular nonlinearities}

\author[B. Bieganowski]{Bartosz Bieganowski}
\address[B. Bieganowski]{\newline\indent
			Faculty of Mathematics, Informatics and Mechanics, \newline\indent
			University of Warsaw, \newline\indent
			ul. Banacha 2, 02-097 Warsaw, Poland}	
			\email{\href{mailto:bartoszb@mimuw.edu.pl}{bartoszb@mimuw.edu.pl}}	

\author[A. Konysz]{Adam Konysz}
\address[A. Konysz]{\newline\indent  	
			Faculty of Mathematics and Computer Science,		\newline\indent 
			Nicolaus Copernicus University, \newline\indent 
			ul. Chopina 12/18, 87-100 Toru\'n, Poland}
			\email{\href{mailto:adamkon@mat.umk.pl}{adamkon@mat.umk.pl}}

\date{}	\date{\today}

\begin{abstract} 
We consider a general elliptic equation
$$
-\Lg u+V(x)u=f(x,u)+g(x,u^2)u
$$
on a closed Riemannian manifold $(M, g)$ and utilize a recent variational approach by Hebey, Pacard, Pollack to show the existence of a nontrivial solution under general assumptions on nonlinear terms $f$ and $g$.
\medskip

   
\end{abstract}

\maketitle

\pagestyle{myheadings} \markboth{\underline{B. Bieganowski, A. Konysz}}{
		\underline{Note on elliptic equations on closed manifolds with singular nonlinearities}}

\section{Introduction}

The aim of this note is to study the existence of solutions to general elliptic problems with singular nonlinearities on a closed (compact and without a boundary) Riemannian manifold $(M,g)$ of dimension $N \geq 3$. Namely, we consider the following equation
\begin{equation}\label{eq:einstein-lichnerowicz}
    -\Lg u+V(x)u=f(x,u)+g(x,u^2)u,
\end{equation}
where $\Lg := \div (\nabla_g)$ is the Laplace-Beltrami operator on $M$, $f:M\times\R\to\R$, and $g:M\times(0,\infty)\to\R$ is the singular term.

One primary motivation for studying such problems arises in general relativity, specifically from the Cauchy problem for the Einstein field equations. In that setting, the so-called Gauss–Codazzi constraint equations must be satisfied by the initial data \cite{Bartnik}. Through the conformal method (see \cite{Choquet-Bruhat, Isenberg}), these constraints reduce to an elliptic equation \eqref{eq:einstein-lichnerowicz} with
\begin{equation} \label{HebeyCase}
f(x,u) = B(x) |u|^{2^*-2}u, \quad g(x,u^2)u = \frac{A(x)}{(u^2)^{2^* / 2} u},
\end{equation}
where $2^* = \frac{2N}{N-2}$ is the critical Sobolev exponent in dimension $N \geq 3$. When also the presence of an electromagnetic field is included, an additional singular term arises and we have (see \cite[Section 7]{IsenbergCQG12})
$$
f(x,u) = B(x) |u|^{2^*-2}u, \quad g(x,u^2)u = \frac{A(x)}{(u^2)^{2^* / 2} u} + \frac{C(x)}{(u^2)^{p / 2} u}
$$
for some $p \in (2,2^*)$.

Here we mention that singular nonlinearities were studied also in the case of a bounded domain $\Omega \subset \R^N$ with Dirichlet boundary conditions in, e.g. \cite{ArcoyaMoreno,BoccardoOrsina}. Since a bounded domain in $\R^N$ cannot be treated as a manifold without boundary, here we only point that we can consider in \eqref{eq:einstein-lichnerowicz} nonlineraties that were considered in \cite{ArcoyaMoreno,BoccardoOrsina}.

We rely on the recent approach in \cite{Hebey} (see also \cite{Premoselli} for further extensions) to outline a set of hypotheses that guarantees the existence of solutions to \eqref{eq:einstein-lichnerowicz} under rather general conditions on $f$ and $g$. We emphasize that the approach is completely based on \cite{Hebey}, adapted to the setting of general nonlinear terms.

We introduce the following assumptions on the regular nonlinear term $f$.
\begin{enumerate}
	\item[(F1)] $f : M \times \R \rightarrow \R$ is H\"older continuous with some exponent $\alpha < 1$ in $x \in M$, and continuous and odd in $u \in \R$; moreover
    $$
    |f(x,u)|\lesssim 1+|u|^{2^*-1} \quad \mbox{ for all } (x,u) \in M \times \R.
    $$
	\item[(F2)] $f(x,u)=o(u)$ as $u\to 0$, uniformly with respect to $x \in M$.
	\item[(F3)] There is $\mu>2$ such that $f(x,u)u\geq \mu F(x,u) \geq 0$, where $F(x,u) := \int_0^u f(x,t) \, dt$.
\end{enumerate}
It is classical to check that (F1), (F2) imply that for every $\delta>0$ there exists $C_\delta>0,$ such that the following inequality holds
\begin{equation}\label{ineq:growth}
    |f(x,u)|\leq \delta |u|+C_\delta |u|^{2^*-1},
\end{equation}
while (F3) is the well-known Ambrosetti-Rabinowitz assumption. On the singular term $g$ we impose the following.
\begin{enumerate}
	\item[(G1)] $g : M \times (0,\infty) \rightarrow \R$ is H\"older continuous with some exponent $\alpha < 1$ in $x \in M$ and continuous in $u \in \R$, $G(x,u)\leq 0$ for all $(x,u) \in M \times \R$, where $G(x,u):=\int_0^u g(x,t)\,dt$.
	\item[(G2)] The map $(0,\infty) \ni u \mapsto G(x,u)$ is increasing for all $x \in M$ and the map $(0,\infty) \ni u \mapsto g(x,u)$ is decreasing for all $x \in M$.
    \item[(G3)]$G(\cdot,u)\in L^1(M)$ for all $u>0$.
    \item[(G4)] $\min_M g(\cdot,u)\to\infty$ as $u\to 0^+$.
\end{enumerate}
\begin{Rem}\label{rem:assumptions}
\begin{enumerate}
    \item[(a)] Note that, since $g$ is continuous in $x$, thanks to (G1), $g(\cdot, u) \in L^\infty (M)$ for every $u > 0$.
    \item[(b)] Since in (G2) we assume that $G (x, \cdot)$ is increasing, we know that $g(x, u) \geq 0$ for all $(x,u) \in M \times (0,\infty)$.
\end{enumerate}
\end{Rem}
On $V$ we assume that
\begin{enumerate}
    \item[(V)] $V \in C^{0,\alpha} (M)$, for some $\alpha<1$, is such that $\inf \sigma (-\Lg + V(x)) > 0$.
\end{enumerate}
In particular, under (V), the operator $-\Lg+V(x)$ on $L^2 (M)$ is coercive, namely there exists a constant $K_V=K(M,g,V)>0$, such that 
$$
\int_M |u|^2\,dv_g\leq K_V\int_M |\nabla u|^2+V(x)u^2\,dv_g
$$
for $u\in H^1(M)$. Hence, we equip the space $H^1(M)$ with norm
$$
\|u\|^2=\int_M |\nabla u|^2+V(x)u^2\,dv_g
$$
that is equivalent to the standard one. We will denote $S_V=S(M,g,V)>0,$ the optimal constant for the embedding
$$
\int_M|u|^{2^*}\,dv_g\leq S_V\left(\int_M |\nabla u|^2+V(x)u^2\,dv_g\right)^{\frac{2^*}{2}}.
$$
Moreover let us assume 
\begin{enumerate}
    \item[(GF)]\label{GF} there exists $\psi\in C^\infty(M)$ such that
\begin{equation}\label{Assumt_exist_phi}
-\int_M G \left(x,\left(\beta \frac{\psi}{\|\psi\|}\right)^2 \right)\,dv_g\leq\frac{1}{2N \left( S_V C_{\frac{1}{4K_V}} \right)^{\frac{N}{2}-1} }
\end{equation}
and
\begin{equation}
\int_M F\left(x,\beta \frac{\psi}{\|\psi\|}\right)\,dv_g>0,
\end{equation}
where 
$$
\beta := \frac{1}{(6(N-1))^\frac12}\left(\frac{1}{2\cdot2^*S_V C_{\frac{1}{4K_V}}}\right)^\frac{N-2}{4},
$$
and $C_{\frac{1}{4K_V}} > 0$ is a constant given in \eqref{ineq:growth} for $\delta = \frac{1}{4K_V}$.
\end{enumerate}
In the case of \eqref{HebeyCase} we recover the assumption from \cite{Hebey}. It is clear that in (GF) we may assume, without loosing generality, that $\|\psi\|=1$.

\begin{Th}\label{th:main}
Suppose that (F1)--(F3), (G1)--(G4), (V), (GF) are satisfied. Then, there exists a nontrivial, positive weak solution $u \in H^1 (M)$ of \eqref{eq:einstein-lichnerowicz}, namely for any $\varphi \in H^1(M)$, $\int_{M}  g(x, u^2) \left| u \varphi \right| \, dv_g < \infty$ and
$$
\int_{M} \nabla_g u \nabla_g \varphi + V(x) u \varphi \, dv_g = \int_{M} f(x,u)\varphi \, dv_g + \int_{M} g(x,u^2) u \varphi \, dv_g.
$$
\end{Th}

\section{The \texorpdfstring{$\varepsilon$}{epsilon}-perturbed problem and the Mountain Pass Theorem}

Define the functional $\cJ_\eps:H^1(M)\to\R$ with formula
$$
\cJ_\eps(u)=\frac12\|u\|^2-\int_M F(x,u)\,dv_g-\frac12 \int_M G(x,\eps+u^2)\,dv_g.
$$
Observe that $\cJ_\varepsilon$ is of $C^1$-class. Indeed, for the first two terms it is standard. Fix $v \in H^1(M)$ and take $t \in (0,1)$, and consider the difference quotient
\begin{align*}
\frac{\frac12 \int_M G(x,\eps+(u+tv)^2)\,dv_g - \frac12 \int_M G(x,\eps+u^2)\,dv_g}{t} &= \frac12 \int_M \frac{G(x,\eps+(u+tv)^2) - G(x,\eps+u^2)}{t} \,dv_g \\
&= \int_{M} g(x, \varepsilon + (u+\theta_t v)^2) (u+\theta_t v) v \, dv_g,
\end{align*}
where in the last equality we used the mean value theorem and $\theta_t \in [0,t]$. To show that the last integral is bounded in $L^1(M)$ uniformly with respect to $t$, it is enough to use the monotonicity of $g$ and the fact that $g(\cdot, \varepsilon) \in L^\infty(M)$.

Following \cite{Hebey}, define for $t>0$ functions $\Phi,\Psi:[0,\infty)\to \R$ by
\begin{align*}
\Phi(t)&= \frac14t^2-S_VC_{\frac{1}{4K_V}}t^{2^*}, \\
\Psi(t)&= \frac34t^2+S_VC_{\frac{1}{4K_V}}t^{2^*},
\end{align*}
then
$$
\Phi(\|u\|)\leq \frac12\|u\|^2-\int_M F(x,u)\,dv_g \leq \Psi(\|u\|).
$$
To simplify the notation we set $C:=C_{\frac{1}{4K_V}}$. Maximum of $\Phi$ is attained in
$$
t_0:=\left(\frac{1}{2\cdot2^*S_VC}\right)^\frac{N-2}{4}.
$$

\begin{Lem}\label{lem1}
There exists $t_1 > 0$ such that
$$
\cJ_\varepsilon (t_1 \psi) < \inf_{\|u\|=t_0} \cJ_\varepsilon (u) \quad \mbox{and} \quad \|t_1 \psi\| < t_0,
$$
where $\psi$ is given in (GF).
\end{Lem}

\begin{proof}
Let 
$$
\theta:=\left( \frac{1}{12(N-1)} \right)^{1/2}.
$$
Then $t_1:=\theta t_0$, and using that $N \geq 3$, we get
\begin{align*}
\Psi(t_1)=\left(\frac{1}{16(N-1)}+\left(\frac{1}{12(N-1)}\right)^{\frac{2^*}{2}}\frac{N-2}{4N}\right)\left(\frac{1}{2\cdot 2^*S_VC}\right)^\frac{N-2}{2}\\
< \left(\frac{1}{8}-\frac12\frac{N-2}{4N}\right)\left(\frac{1}{2\cdot 2^*S_VC}\right)^\frac{N-2}{2}=\frac{1}{2}\Phi(t_0).
\end{align*}
Note that \eqref{Assumt_exist_phi} takes a form
\begin{equation}\label{Assumpt_phi_Phi}
-\frac12\int_M G \left(x,(t_1\psi)^2\right)\,dv_g\leq \frac12\Phi(t_0),
\end{equation}
where we used that $\|\psi\|=1$. Then, by \eqref{Assumpt_phi_Phi} and monotonicity of $G$, we have that for any $\|u\|=1$,
\begin{align*}
\cJ_\eps \left(t_1\psi\right)&=\frac12\|t_1\psi\|^2-\int_M F(x,t_1\psi)\,dv_g-\frac12\int_M G(x,\eps+(t_1\psi)^2)\,dv_g\\
&\leq \Psi(t_1) - \frac12 \int_M G(x,\eps+(t_1\psi)^2)\,dv_g \\
&< \frac12 \Phi(t_0)- \frac12 \int_M G(x,(t_1\psi)^2)\leq \Phi(t_0)
\leq \frac12\|t_0 u\|^2-\int_M F(x,t_0 u)\,dv_g
\leq \cJ_\eps(t_0 u).
\end{align*}
Hence
$$
\cJ_\varepsilon (t_1 \psi) < \inf_{\|u\| = t_0} \cJ_\varepsilon(u) \quad \mbox{and} \quad \|t_1 \psi\| < t_0,
$$
and the proof is completed. 
\end{proof}

\begin{Lem}\label{lem2}
$$
\lim_{t\to\infty}\cJ_\eps(t\psi)=-\infty.
$$ 
\end{Lem}

\begin{proof}
The condition (F3) implies that $F(u)\geq |u|^\mu$ for every $u \in \R$, we get
\begin{align*}
\cJ_\eps(t\psi)&\leq \frac12t^2-t^\mu\int_M |\psi|^\mu\,dv_g - \frac12 \int_M G\left(x, \varepsilon + (t\psi)^2 \right) \,dv_g \\
&\leq \frac12t^2-t^\mu\int_M |\psi|^\mu\,dv_g - \frac12 \int_M G\left(x, \varepsilon \right) \,dv_g
\end{align*}
and we have following limit
$$
\lim_{t\to\infty}\cJ_\eps(t\psi)=-\infty.
$$
\end{proof}

Thanks to Lemma \ref{lem2}, we can find $t_2>t_0$ such that $\cJ_\eps(t_2\varphi)<0$. Define
$$
\Gamma=\{ \gamma\in C([0,1]; H^1(M)):\gamma (0)=t_1\varphi,\gamma (1)=t_2\varphi \}.
$$
From Lemmas \ref{lem1} and \ref{lem2}, as in \cite{Hebey}, using Mountain Pass Theorem we can find a Palais-Smale sequence on the level 
\begin{equation} \label{eq:c-eps}
c_\eps:=\inf_{\gamma\in\Gamma}\max_{t\in [0,1]}\cJ_\eps(\gamma(t)) \geq \Phi(t_0) > 0,
\end{equation}
i.e.
\begin{equation} \label{eq:PS-seq}
\cJ_\eps(u_n)\to c_\eps \quad\hbox{and}\quad \cJ_\eps'(u_n)\to 0.
\end{equation}
Moreover, since $\cJ_\varepsilon$ is even, we may assume that $u_n \geq 0$ almost everywhere on $M$.

\begin{Prop}\label{prop:eps}
Up to a subsequence, $(u_n)$ converges weakly in $H^1(M)$ and almost everywhere to a weak, nonnegative solution $u_\varepsilon \in H^1(M)$ of the problem
$$
-\Lg u+V(x)u=f(x,u)+g(x,\eps+u^2)u.
$$
\end{Prop}

\begin{proof}
We can rewrite \eqref{eq:PS-seq}
\begin{equation}\label{PS-sequence_J}
\cJ_\eps(u_n)=\frac12\|u_n\|^2-\int_M F(x,u_n)\,dv_g- \frac12 \int_M G(x,\eps+u_n^2)\,dv_g=c_\eps+o(1)
\end{equation}
and
\begin{equation}\label{PS-sequence_J'}
\|u_n\|^2-\int_M f(x,u_n)u_n\,dv_g-\int_M g(x,\eps+u_n^2)u_n^2\,dv_g= \cJ_\varepsilon '(u_n)(u_n) = o(\|u_n\|).
\end{equation}
Combining these two formulas, in the same way as in \cite[Proof of Theorem 3.1]{Hebey}, we obtain that
\begin{align*}
2c_\eps+o(\|u_n\|)&\geq \int_M f(x,u_n)u_n-2F(x,u_n)\,dv_g+\underbrace{\int_M g(x,\eps + u_n^2)u_n^2-G(x,\eps+u_n^2)\,dv_g}_{\geq 0}\\
&\geq (\mu-2)\int_M F(x,u_n)\,dv_g,
\end{align*}
where (F3), (G1) and Remark \ref{rem:assumptions}(b) were used. Using this inequality and \eqref{PS-sequence_J} we get that
\begin{equation}\label{est:norm}
    \|u_n\|^2\leq \frac{4c_\eps}{\mu-2}+2c_\eps+o(\|u_n\|)
\end{equation}
for sufficiently large $n$, so the Palais-Smale sequence is bounded and up to a subsequence we have following covergences:
\begin{align*}
u_n\weakto u_\eps &\quad \mbox{ in } H^1 (M)\\
%u_n\to u_\eps &\quad \mbox{ in } L^p(M)\quad \mbox{for some } p \in (2,2^*)\\
u_n\to u_\eps &\quad \mbox{ a.e. in } M .
\end{align*}
Fix any test function $\varphi \in C_0^\infty (M)$. Take any measurable set $E \subset M$ and note that
\begin{align}\label{est:convergence f}
\begin{split}
\int_E |f(x,u_n)\varphi|\,dv_g&\lesssim  \int_E (1+|u_n|^{2^*-1})|\varphi|\,dv_g\lesssim |\varphi\chi_E|_1+\int_E|u_n|^{2^*-1}|\varphi|\,dv_g\\
&\lesssim |\varphi\chi_E|_1+\left( S_V^{1/2^*}\|u_n\| \right)^{2^*-1}|\varphi\chi_E|_{2^*}
\end{split}
\end{align}
and since $(u_n)$ is bounded in $H^1(M)$, the family $\{ f(\cdot, u_n) \varphi\}$ is uniformly integrable and from the Vitali convergence theorem,
\begin{equation}\label{eq:f-conv}
\int_M f(x,u_n)\varphi\,dv_g\to\int_M f(x,u_\eps)\varphi\,dv_g.
\end{equation}
To pass to the limit in the singular term, (G2) and Cauchy-Schwarz inequality yield
\begin{equation}\label{est:convergence g}
\int_E g(x,\eps+u_n^2) |u_n\varphi|\,dv_g\leq |g(\cdot,\eps)|_\infty\int_M \chi_E| u_n\varphi|\,dv_g\leq |g(\cdot,\eps)|_\infty |\varphi\chi_E|_2|u_n|_2,
\end{equation}
having in boundedness of $(u_n)$ in $H^1(M)$, we get that $\{g(\cdot, \varepsilon + u_n^2) u_n \varphi\}$ is uniformly integrable and from Vitali convergence theorem 
\begin{equation}\label{eq:g-conv}
\int_M g(x,\eps+u_n^2) u_n \, dv_g \to \int_M g(x,\eps+u_\eps^2) u_\eps \, dv_g.
\end{equation}

Summing up, from weak convergence of $u_n$, \eqref{eq:f-conv}, and \eqref{eq:g-conv} we can pass to the limit in the condition $\cJ'(u_n)(\varphi)=0$ and we find that $u_\eps$ is a weak solution of the problem
\begin{equation}\label{eq:eps}
-\Lg u+V(x)u=f(x,u)+g(x,\eps+u^2)u.
\end{equation}
Since $u_n \geq 0$, from the pointwise convergence, $u_\varepsilon \geq 0$.
\end{proof}

\section{Regularity of solutions to \texorpdfstring{$\varepsilon$}{epsilon}-perturbed problem (\ref{eq:eps})}

In order to pass with $\eps\to 0^+$, following the strategy of \cite{Hebey}, we need information about the regularity of the solutions.
\begin{Prop}
The nonnegative, weak solution $u_\varepsilon \in H^1 (M)$ found in Proposition \ref{prop:eps} is of class $C^{2,\alpha}(M)$ for some $\alpha < 1$, and $u_\varepsilon > 0$ everywhere on $M$.
\end{Prop}

\begin{proof}
Fix $\eps>0$. In the equation \eqref{eq:eps} we denote by
$$
h(x):=V(x)-g(x,\eps+u_\eps^2), \ x \in M,
$$
and observe that $h\in L^\infty(M).$ Indeed
$$
|h| = |V-g(\cdot ,\eps+u_\eps^2)|\leq |V|+g(\cdot,\eps+u_\eps^2)\leq |V|+g(\cdot,\eps)\in L^\infty(M).
$$
Denote now $w:=u_\eps$. From the strong maximum principle, we get that $w > 0$. Let us rewrite the equation \eqref{eq:eps} in the form
$$
-\Lg u_\eps=-hu_\eps+\frac{f(x,w(x))}{w}u_\eps
$$
and denote
$$
k(x,u_\eps):=-h(x)u_\eps+\frac{f(x,w)}{w(x)}u_\eps.
$$
Now the equation \eqref{eq:eps} takes form
$$
-\Lg u_\eps=k(x,u_\eps).
$$
From \eqref{ineq:growth}, for every $\delta > 0$ we can find $C_\delta > 0$ such that
$$
\left|\frac{f(x,w)}{w}\right|\leq \delta+C_\delta |w|^{2^*-2}.
$$
So we get that 
$$
|k(x,u_\eps)|\leq \left(\underbrace{|h(x)|+ \left|\frac{f(x,w(x))}{w(x)}\right|}_{a(x):=}\right)(1+|u_\eps|)
$$
and also $a\in L^{\frac{N}{2}}(M).$ So by the Brezis-Kato type result (see Lemma \ref{Brezis-Kato-lemma}), we get that $u_\eps\in L^q(M)$ for every $q<\infty$, and the standard bootstrap procedure shows that $u_\eps\in W^{2,q}(M).$ Now let us fix $\alpha<1$ and choose $q$ such that $\alpha\leq 1- \frac{N}{q}$. Then using the Sobolev embedding theorem (see \cite[Theorem 2.10, Theorem 2.20]{Aubin}) we get $u_\eps\in C^{1,\alpha}(M)$. Then, clearly $u_\varepsilon \in L^\infty (M)$ and it is easy to see that the map $M \ni x \mapsto k(x,u_\varepsilon(x)) \in \R$ is $C^{0,\alpha} (M)$. Hence, in particular, $u_\varepsilon \in W^{2,2} (M) \cap L^\infty (M)$ and $\Delta_g u \in C^{0,\alpha}(M)$. Then, the elliptic regularity theory yields that $u_\varepsilon \in C^{2,\alpha} (M)$. 
\end{proof}

\section{Proof of Theorem \ref{th:main}}
Similarly as in \cite{Hebey}, considering in \eqref{eq:c-eps} the path $\gamma(t)=t\psi, t\in[t_1,t_2]$ we get that
\begin{equation}\label{est:c_eps}
    0<\Phi(t_0)\leq c_\eps\leq c := \sup_{t \in [t_1, t_2]} \cJ(t\psi).
\end{equation}
Let $(\varepsilon_k) \subset (0,\infty)$ be a sequence such that $\eps_k\to 0^+$, and denote $u_k:=u_{\eps_k}$. Observe that by \eqref{est:norm}, \eqref{est:c_eps} and the weak lower semicontinuity of the norm, we get that sequence $(u_k)$ is bounded in $H^1(M)$ and, up to passing to a subsequence,
\begin{align*}
u_k\weakto u &\quad \mbox{ in } H^1(M)\\
%u_k\to & u \quad \hbox{in } L^p(M)\quad \hbox{for some } 2^*>p>2\\
u_k\to u &\quad \mbox{ a.e. in } M .
\end{align*}
Arguing similarly as in \eqref{est:convergence f} we get that
\begin{equation}\label{convergence of f on solution}
\int_M f(x,u_k)\varphi\,dv_g\to\int_M f(x,u)\varphi\,dv_g
\end{equation}
for every $\varphi\in C^\infty_0(M).$
Now we have to show
$$
\int_M g(x,\eps_k+u_k^2)u_k\varphi\,dv_g\to \int_M g(x,u^2)u\varphi\,dv_g.
$$
Firstly, we will show that there exists $\delta_0$ such that $u_k\geq \delta_0$ for $k$ sufficiently large. Let $x_k\in M$ be the point where $u_k$ has a global minimum. Then obviously $$-\Lg u (x_k)\leq 0$$ and we obtain
\begin{equation}\label{eq:nier}
V(x_k)u_k(x_k)+|f(x_k,u_k(x_k))|\geq g(x_k,\eps_k+u_k(x_k)^2)u_k(x_k).
\end{equation}
Suppose by contradiction that $u_k(x_k) \to 0$. Then \eqref{eq:nier}, (F2) and (G4) imply that
\begin{align*}
\max_M V + o(1) \geq V(x_k)+\frac{|f(x_k,u_k(x_k))|}{u_k(x_k)}\geq g(x_k,\eps_k+u_k(x_k)^2) \geq \min_M g(\cdot, \eps_k+u_k(x_k)^2) \to \infty 
\end{align*}
as $k\to\infty$, which is a contradiction. It follows that
$$
\min_M u_k \geq \delta_0
$$
for some $\delta_0 > 0$. So we can estimate
$$
g(x,\eps_k+u_k^2)\leq g(x,\delta_0^2),
$$
then using the H\"older inequality
$$
\int_E g(x,\eps+u_k^2)| u_k\varphi |\,dv_g\leq |g(x,\delta_0^2)|_\infty|u_k|_2 |\chi_E \varphi|_2,
$$
so by the boundedness of $(u_k)$ in $L^2 (M)$, we get by the Vitali convergence theorem that
$$
\int_M g(x,\eps+u_k^2)u_k\varphi\,dv_g\to \int_M g(x,u^2)u\varphi\,dv_g
$$
holds. Hence, letting $k\to \infty$ in
$$
-\Lg u_k+V(x)u_k=f(x,u_k)+g(x,\eps_k+u_k^2)u_k.
$$
we obtain that $u$ is a weak solution of \eqref{eq:einstein-lichnerowicz}. In particular, $u$ is positive a.e., since - from the pointwise convergence, $u(x) \geq \delta_0$ for a.e. $x \in M$.

\appendix
\section{Brezis-Kato result on a compact Riemannian manifold}

In the appendix we present a well-known Brezis-Kato result, see e.g. \cite[Lemma B.3]{Struwe}. Since we were unable to find a reference for the statement in the case of a Riemannian manifold, we provide it here (based on the proof of \cite[Lemma B.3]{Struwe}) for the readers' convenience. 

\begin{Lem}\label{Brezis-Kato-lemma}
Let $u\in H^1(M)$ be a weak solution to the equation
\begin{equation}\label{eq:Brezis-Kato}
-\Lg u=g(x,u),
\end{equation}
where $g:M\times\R\to \R$ is a Carath\'eodory function satisfying
$$
|g(x,u)|\leq a(x)(1+|u|)
$$
for some $a\in L^{N/2}(M)$. Then $u\in L^q(M)$ for any $q<\infty.$
\end{Lem}
\begin{proof}
    Let $s\geq 0,$ $L\geq 1$ and denote $\varphi=\varphi_{s,L}:=u\min\{|u|^{2s},L^2\}\in H^1(M)$. Observe that
    $$
    \int_M \nabla u\cdot \nabla \varphi\,dv_g = \int_M |\nabla u|^2\min \{|u|^{2s},L^2\}\,dv_g +\frac{s}{2}\int_{\{x\in M:|u(x)|^s\leq L\}} |\nabla(|u|^2)|^2|u|^{2s-2}\,dv_g.
    $$
    Hence, testing equation \eqref{eq:Brezis-Kato} with $\varphi$ we get that 
\begin{align*}
    &\quad \int_M |\nabla u|^2\min \{|u|^{2s},L^2\}\,dv_g +\frac{s}{2}\int_{\{x\in M:|u(x)|^s\leq L\}} |\nabla(|u|^2)|^2|u|^{2s-2}\,dv_g = \int_M \nabla u\cdot \nabla \varphi\,dv_g= \\
    &= \int_M g(x,u) \varphi \, dv_g \leq \int_M a(1+|u|)|u| \min \{|u|^{2s},L^2\}\,dv_g\\
    &\leq \int_M a(1+2|u|^2) \min \{|u|^{2s},L^2\}\,dv_g =  \int_M a\min\{|u|^{2s},L^2\}\,dv_g+2\int_M a|u|^2 \min \{|u|^{2s},L^2\}\,dv_g\\
    &=  \int_M a\min\{|u|^{2s},L^2\}(1-|u|^2)\,dv_g+3\int_M a|u|^2 \min \{|u|^{2s},L^2\}\,dv_g\\
    &\leq \int_M a\,dv_g +  3\int_M a|u|^2\min\{|u|^{2s},L^2\}\,dv_g.
\end{align*}
Then, assuming that $u\in L^{2s+2}(M)$, for any $K\geq 1$ we can estimate ($\widetilde{C}>0$ may vary from one line to another):
\begin{align*}
&\quad \int_M |\nabla(u\min\{|u|^{s},L\})|^2\,dv_g\leq 2\int_M |\nabla u|^2\min\{|u|^{2s},L^2\}\,dv_g+2\int_{\{x\in M:|u(x)|^s\leq L\}} |u\nabla (|u|^{s})|^2\,dv_g\\
&\leq \widetilde{C} \left(1+\int_M a|u|^2\min\{|u|^{2s},L^2\}\,dv_g\right)\\
&\leq \widetilde{C} \left(1+K\int_M |u|^2\min\{|u|^{2s},L^2\}\,dv_g+\int_{\{x\in M:a(x)>K\}} a|u|^2\min\{|u|^{2s},L^2\}\,dv_g\right)\\
&\leq \widetilde{C} \left(1+K|u|_{2s+2}^{2s+2}+\int_{\{x\in M:a(x)>K\}} a|u|^2\min\{|u|^{2s},L^2\}\,dv_g\right)\\
&\leq \widetilde{C}(1+K)+\underbrace{\widetilde{C}\left(\int_{\{x\in M:a(x)>K\}} a^{N/2}\,dv_g\right)^{2/N}}_{=:\gamma(K)}\left(\int_M (|u|\min\{|u|^s,L\})^{2^*}\,dv_g\right)^{2/2^*}.
\end{align*}
Now let us choose $K\geq 1$ such that $\gamma(K)\leq\frac12$, and we obtain that
$$
\int_M |\nabla(u\min\{|u|^{s},L\})|^2\,dv_g\lesssim 1,
$$
so we have uniform bound (with respect to $L$) on the $L^2$-norm of $\nabla(u\min\{|u|^{s},L\})$. Hence taking $L\to\infty$ we obtain that 
$$
\int_M |\nabla(|u|^{s+1})|^2<\infty.
$$
Thus we have shown that $|u|^{s+1}\in H^1(M)\subset L^{2^*}(M).$ That means that $u\in L^{\frac{2(s+1)N}{N-2}}$. Taking $s_0=0$ and $s_i+1:=(s_{i-1}+1)\frac{N}{N-2}$, we obtain $u\in L^q(M)$ for every $q<\infty$.
\end{proof}

\section*{Acknowledgements}
Bartosz Bieganowski and Adam Konysz were partly supported by the National
Science Centre, Poland (Grant no. 2022/47/D/ST1/00487).

\end{document}